\documentclass[11pt]{article}
\usepackage{amssymb}
\usepackage{amsfonts}
\usepackage{amscd}
\usepackage{epsfig}
\usepackage{amsmath,amstext,amsthm,amsbsy,amssymb}

\newtheorem{thm}{Theorem}[section]
\newtheorem{defi}{Definition}[section]
\newtheorem{lem}{Lemma}[section]
\newtheorem{pro}{Proposition}[section]
\newtheorem{cor}{Corollary}[section]

\newcommand{\bc}{{\mathbb C}}
\newcommand{\br}{{\mathbb R}}
\newcommand{\bh}{{\mathbb H}}

\newcommand{\fq}{{\mathcal Q}}
\newcommand{\bF}{{\mathbb F}}
\newcommand{\bP}{{\mathbb P}}
\newcommand{\bJ}{{\mathbb J}}
\newcommand{\mtx}[4]

\begin{document}

\title{On volumes of  quaternionic  hyperbolic n-orbifolds}
\author{ Wensheng Cao \quad  Jianli Fu\\
School of Mathematics and Computational Science,\\
Wuyi University, Jiangmen, Guangdong 529020, P.R. China\\
e-mail: {\tt wenscao@aliyun.com}\\
}
\date{}
\maketitle
\bigskip
{\bf Abstract:} By use of  H. C. Wang's bound on the radius of a ball embedded in the fundamental domain of a lattice of a semisimple Lie group, we construct an explicit lower bound for the volume of a quaternionic hyperbolic orbifold that depends only on dimension.
\smallskip

\medskip

\medskip


\section{Introduction}
Let $\bF$ be  real numbers  $\br$,  complex numbers $\bc$ or quaternions $\bh$,  and ${\bf H}_\bF^n$  the $n$-dimensional hyperbolic space over $\bF$.  Let $G$ be  the linear groups which act as the isometries in ${\bf H}_\bF^n$. For $\bF=\br$, $\bc$ and $\bh$,  $G$  are usually  denoted by ${\rm O}(n,1)$, ${\rm U}(n,1)$ and ${\rm Sp}(n,1)$, respectively.   A hyperbolic orbifold $\fq$  is a quotient  of  ${\bf H}_\bF^n$  by  a discrete subgroup $\Gamma$ of  $G$.  An orbifold ${\bf H}_\bF^n/\Gamma$ is a manifold
when $\Gamma$ contains no elements of finite order.

In 1945, Siegel\cite{sie43,sie45} posed the problem of identifying the minimal
covolume lattices of isometries of real hyperbolic $n$-space, or more generally rank-1 symmetric spaces.
He solved the problem in two dimensions and Gehring and Martin \cite{gema} solved similar problem in three dimensions.  For the general cases, the results  of Wang and Gunther  \cite{gal90,wan69, wan72} imply that  the hyperbolic volumes of $\fq$  form a discrete subset of $\br$.

Firedland and Hersonsky  \cite{frhe93} constructed a lower
bound for $r_n$, the largest number such that every hyperbolic $n$-manifold contains a round ball of
that radius.  From this one can compute an explicit lower bound for the volume of a hyperbolic $n$-manifold. Recently,  some  analogous results  have been obtained in complex and quaternionic settings \cite{cao16,xie14}

Adeboye  obtained an explicit lower bound for the volume of a real hyperbolic orbifold  depending on the dimension \cite{abe08}. The main tool is the spectral radius of the involved matrices.   Such a technique  was employed latter in complex and quaternionic  settings \cite{fu11,haxi16}.

Recently Adeboye and Wei reconsidered  the question of lower bound for the volume of a real hyperbolic orbifold  with the tools  of Lie group and  Riemannian  submersion \cite{abe12}. They obtained the following theorem.

  \begin{thm} {\rm \bf (Theorem 0.1 of  Adeboye and Wei \cite{abe12})}  \label{rabe12}
The volume of a real hyperbolic $n$ -orbifold is bounded below by $\mathcal{R}(n)$,  an  explicit constant depending only on dimension, given by
$$ \mathcal{R}(n)=\frac{2^{\frac{6-n}{4}}\pi^{\frac{n}{4}}(n-2)!(n-4)!\cdots 1}{(2+9n)^{\frac{n^2+n}{4}}
\Gamma(\frac{n^2+n}{4})} \int^{\min[0.0806\sqrt{2+9n},\pi]}_0(\sin\rho)^{\frac{n^2+n-2}{2}} d\rho.$$
\end{thm}

    Such work  significantly improved  upon the volume bounds of \cite{abe08,mar89}.
   The authors also obtained the following result in complex setting.
 \begin{thm}{\rm \bf (Theorem 0.1 of  Adeboye and Wei \cite{abe14})}\label{rabe14}
The volume of a complex  hyperbolic $ n$ -orbifold is bounded below by $\mathcal{C}(n)$,  an  explicit constant depending only on dimension, given by
$$ \mathcal{C}(n)=\frac{2^{n^2+n+1}\pi^{\frac{n}{2}}(n-1)!(n-2)!\cdots!3!2!1!}{(36n+21)^{\frac{n^2+2n}{2}}
\Gamma(\frac{n^2+2n}{2})} \int^{\min[0.06925\sqrt{36n+21},\pi]}_0(\sin\rho)^{n^2+2n-1} d\rho.$$
\end{thm}

As interest in quaternionic hyperbolic space has grown,  many results from  real and complex hyperbolic geometry  have been carried over to  the quaternionic  arena (see \cite{capa, cao16, kimp} et al). Due to the noncommutativity  of quaternions, the analogous problems  in quaternionic setting   are sometimes more complicated.

  Motivated by the ideas  of  Adeboye and Wei in \cite{abe12,abe14}, we  will consider the question of lower bound for the volume of a quaternionic hyperbolic orbifold  with the tools of  Lie group and  Riemannian  submersion.

We will  construct a   Riemannian submersion from the  quotient ${\rm Sp}(n,1)/\Gamma$ to the quotient  ${\bf H}_\bh^n/\Gamma$.  With this Riemannian submersion, we can employ Wang's result \cite[Theorem 5.2]{wan69} to produce an inscribed ball of radius $\frac{R_{{\rm Sp}(n,1)}}{2}$ in ${\bf H}_\bh^n/\Gamma$ and obtain the lower bound by a comparison theorem of Gunther \cite[Theorem 3.101]{gal90}.

   Our main result is the  following theorem.
 \begin{thm}\label{mainthm}
The volume of a quaternionic  hyperbolic $ n$ -orbifold is bounded below by $\mathcal{Q}(n)$,  an  explicit constant depending only on dimension, given by
$$ \mathcal{Q}(n)=\frac{\pi^{\frac{3n}{2}}(2n+1)!(2n-1)!\cdots!5!3!1!}{2^{n-1}
\Gamma(\frac{2n^2+5n+3}{2})\Gamma(\frac{4n+1}{2}) (\frac{3+4\sqrt{2}}{2})^{\frac{2n^{2}+5n+3}{2}}} \int^{0.2372}_0(\sin\rho)^{2n^2+5n+2} d\rho.$$
\end{thm}

As in \cite{abe12,abe14}, the volume bounds for hyperbolic orbifolds provide  information
on the order of the symmetry groups of hyperbolic manifolds. Following Hurwitz's
formula for groups acting on surfaces, we have the following corollary.

\begin{cor}
Let $M$ be a quaternionic hyperbolic n-manifold. Let $H$ be a group of isometries
of $M$. Then
$$|H|\leq \frac{{\rm Vol}(M)}{\mathcal{Q}(n)}.$$
\end{cor}

 The paper is organized  as follows.  Section \ref{qhs} contains some  necessary background material for quaternionic hyperbolic geometry.  In Section \ref{liespn1}, we will present the Cartan  decomposition of
 $\mathfrak{sp}(n,1)$ and  obtain the standard $\br$-vector space basis for it. Also we will define the canonical metric in Lie group ${\rm Sp}(n,1)$.   Section \ref{sectcur} aims to obtain the bound of the sectional curvatures of
 ${\rm Sp}(n,1)$  with respect to the scaled  canonical metric.  In order to obtain better estimate, we will use some formulae of the connection  and curvature which are slight different from those in \cite{abe14}.
 Section \ref{volume} contains the proof of Theorem \ref{mainthm} with the similar route map employed  in \cite{abe14}. In Section \ref{realcomplex},  we  reestimate the bound of the sectional curvatures of ${\rm SO_o}(n,1)$ and
${\rm SU}(n,1)$.  These new bounds imply  slight improvements  of the results in \cite{abe12,abe14}.

\section{Quaternionic hyperbolic space}\label{qhs}

 In this section, we give some necessary background
materials of quaternionic  hyperbolic geometry. More details can be
found in \cite{chen,kimp,most}.

We recall that a real quaternion is of the form $q=q_0+q_1{\bf i}+q_2{\bf j}+q_3{\bf k}\in \bh$
where $q_i\in \br$ and $ {\bf i}^2 = {\bf j}^2 = {\bf k}^2 = {\bf
i}{\bf j}{\bf k} = -1.$ Let $\overline{q}=q_0-q_1{\bf i}-q_2{\bf
j}-q_3{\bf k}$ and $|q|= \sqrt{\overline{q}q}=\sqrt{q_0^2+q_1^2+q_2^2+q_3^2}$  be the  conjugate  and modulus of $q$, respectively.

Let $\bh^{n,1}$ be the vector space of dimension n+1 over $\bh$ with
the unitary structure defined by the Hermitian form
$$
\langle{\bf z},\,{\bf w}\rangle={\bf w}^*J{\bf z}=
\overline{w_1}z_1+\cdots+\overline{w_n}z_n-\overline{w_{n+1}}z_{n+1},
$$
where ${\bf z}$ and ${\bf w}$ are the column vectors in $\bh^{n,1}$ with
entries $(z_1,\cdots,z_{n+1})$ and $(w_1,\cdots,w_{n+1})$
respectively, $\cdot^*$ denotes the conjugate transpose and $J$ is
the Hermitian matrix
$$J=\left(
      \begin{array}{cc}
        I_n & 0 \\
         0 & -1 \\
      \end{array}
    \right).$$
We define a  unitary transformation $g$ to be an automorphism
of $\bh^{n,1}$, that is, a linear bijection such that $\langle
g({\bf z}),\,g({\bf w})\rangle=\langle{\bf z},\,{\bf w}\rangle$ for
all ${\bf z}$ and ${\bf w}$ in $\bh^{n,1}$. We denote the group of all
unitary transformations by ${\rm Sp}(n,1)$, which  is the noncompact Lie group
\begin{equation}{\rm Sp}(n,1)=\{A\in {\rm {\rm GL}}(n+1,\bh): A^*JA=J\}.\end{equation}
Let
\begin{eqnarray*}
V_{-} &  = & \bigl\{{\bf z} \in \bh^{n,1}:\langle{\bf z},\,{\bf
z}\rangle<0\bigr\}.
\end{eqnarray*}
It is obvious that $V_-$ is invariant under $ {\rm Sp}(n,1)$.
Let $$\bP:\bh^{n,1}-\{0\}\to
\bh\mathbf{P}^n$$ be the canonical projection onto quaternionic projective space.  Quaternionic hyperbolic $n$-space,  ${\bf H}_\bh^n$, is defined to be the space $\bP(V_-)$ together the  Bergman metric. The Bergman metric  on ${\bf H}_\bh^n$ is given by the distance formula
\begin{equation*}
\cosh^2\frac{\rho(z,w)}{2}=\frac{\langle{\bf z},\,{\bf
w}\rangle \langle{\bf w},\,{\bf z}\rangle}{\langle{\bf z},\,{\bf
z}\rangle \langle{\bf w},\,{\bf w}\rangle},\ \ \mbox{where}\ \ {\bf
z}\in \bP^{-1}(z),{\bf w}\in \bP^{-1}(w).
 \end{equation*}
  The holomorphic isometry group of ${\bf H}_\bh^n$ with respect to the
Bergman metric is the projective unitary group ${\rm PSp}(n, 1)={\rm Sp}(n,1)/\pm I_{n+1}$ and acts
on $\bP(\bh^{n,1})$ by matrix multiplication.

Let $${\rm Sp}(n)=\{A\in GL(n,\bh):AA^*=I_n\}.$$
Since the  stabilizer of the point of  ${\bf H}_\bh^n$ with the homogeneous coordinates $(0,\cdots,0,1)$ is
$${\rm Sp}(n)\times {\rm Sp}(1)=\left\{ \left(
                                                \begin{array}{cc}
                                                  A & 0 \\
                                                  0 & q \\
                                                \end{array}
                                              \right): A\in {\rm Sp}(n),q\in {\rm Sp}(1) \right\},$$
we have the following identification
\begin{equation}{\bf H}_\bh^n={\rm Sp}(n,1)/{\rm Sp}(n)\times {\rm Sp}(1).\end{equation}

\section{ The Lie group ${\rm Sp}(n,1)$}\label{liespn1}
This section contains  some necessary materials of  the Lie group ${\rm Sp}(n,1)$ including the  Cartan  decomposition
 and the standard $\br$-vector space basis  of  the Lie algebra $ \mathfrak{sp}(n,1)$,  and  the canonical metric in Lie group ${\rm Sp}(n,1)$.

\subsection{The  Cartan  decomposition of  $ \mathfrak{sp}(n,1)$}
A matrix Lie group is a closed subgroup of  $ {\rm GL}(n,\mathbb{H}) $. Recall that for a square matrix $ X $,
$$ e^X=I+X+\frac{1}{2}X^2+\cdots. $$
The  Lie algebra of a matrix Lie group $G$ is a vector space, defined as the set of matrices $ X $ such that $e^{tX}\in G $, for all real numbers $ t $. The Lie algebra of ${\rm GL}(n,\mathbb{H}) $, denoted by $ \mathfrak{gl}(n,\mathbb{H}) $, is the set of $ n\times n $ matrices over $\mathbb{H}$.

The Lie algebra of $ {\rm Sp}(n,1) $ is defined and denoted by
$$ \mathfrak{sp}(n,1)=\{X \in \mathfrak{gl}(n+1,\mathbb{H}):\ \  JX^*J=-X\}.$$
The fixed point set of the Cartan involution $\theta(X)=JXJ$ is a maximal compact subgroup $K$ of ${\rm Sp}(n,1)$ isomorphic to ${\rm Sp}(n) \times {\rm Sp}(1)$.
The corresponding standard Cartan decomposition $$\mathfrak{sp}(n,1)=\mathfrak{k}+\mathfrak{p}$$ is given by
$$\mathfrak{k}=\left\{\left(\begin{array}{cc}
                   M & 0 \\
                   0 & q
                 \end{array}\right): M\in \mathfrak{sp}(n),q\in \mathfrak{sp}(1)\right\},$$
$$\mathfrak{p}=\left\{\left(\begin{array}{cc}
                   0 & Z \\
                   Z^* & 0
                 \end{array}\right): Z\in\bh^n\right\},$$
where $$\mathfrak{sp}(n)=\{M \in  \mathfrak{gl}(n,\mathbb{H}):   M+M^*=0\}.$$
The Lie bracket of a matrix Lie algebra is determined by matrix operations
$$ [X,Y]=XY-YX. $$

\begin{defi}\label{preli}
For each $ n $, let $ e_{jk} \in \mathfrak{gl}(n+1,\mathbb{H}) $ be the matrix with $ 1 $ in the $jk$ -position and  $0$ elsewhere.  We define \begin{equation}\label{albeij} \alpha_{jk}=(e_{jk} -e_{kj}),\ \ \beta_{jk}=(e_{jk}+e_{kj}).\end{equation}
\end{defi}
The following proposition describes the Lie bracket of $ \mathfrak{sp}(n,1)$.  The proof involves straightforward calculation form the fact $$e_{ij}e_{kl}=\delta_{jk}e_{il}$$ and the definitions of $\alpha_{ij}$ and $\beta_{ij}$ and therefore is omitted.

\begin{pro}\label{baseprd}{\rm  (cf. \cite[Proposition 2.2]{abe14})}
Let ${\bf I}_1={\bf i}, {\bf I}_2={\bf j}$ and  ${\bf I}_3={\bf k}$.  For $ 1\leq j<k\leq n,1\leq l<m\leq n$, we have the following equalities:
\begin{equation}\label{alal}[\alpha_{jk},\alpha_{lm}]=\delta_{kl}\alpha_{jm}+\delta_{km}\alpha_{lj}+\delta_{jm} \alpha_{kl}+\delta_{lj}\alpha_{mk} , \end{equation}
\begin{equation}\label{albeI}[\alpha_{jk},{\bf I}_t\beta_{lm}]={\bf I}_t(\delta_{kl}\beta_{jm}+\delta_{km}\beta_{jl}-\delta_{jm} \beta_{kl}-\delta_{lj}\beta_{km}) , t=1,2,3, \end{equation}
\begin{equation}\label{aleI}[\alpha_{jk},{\bf I}_te_{ii}]={\bf I}_t(\delta_{ki}\beta_{ji}-\delta_{ji}\beta_{ki}),t=1,2,3,  \end{equation}
\begin{equation} [\alpha_{jk},\beta_{l,n+1}]=\delta_{lk}\beta_{j,n+1}-\delta_{jl}\beta_{k,n+1}, \end{equation}
\begin{equation} [\alpha_{jk},{\bf I}_t\alpha_{l,n+1}]={\bf I}_t(\delta_{lk}\alpha_{j,n+1}-\delta_{lj}\alpha_{k,n+1}),t=1,2,3, \end{equation}
\begin{equation} [{\bf I}_t\beta_{jk},{\bf I}_t\beta_{lm}]=-(\delta_{kl}\alpha_{jm}+\delta_{km}\alpha_{jl}+\delta_{jm} \alpha_{kl}+\delta_{lj}\alpha_{km}), t=1,2,3, \end{equation}
\begin{equation} [{\bf I}_t\beta_{jk},{\bf I}_s\beta_{lm}]={\bf I}_t{\bf I}_s(\delta_{kl}\beta_{jm}+\delta_{km}\beta_{jl}+\delta_{jm} \beta_{kl}+\delta_{lj}\beta_{km}) , t\neq s, \end{equation}
\begin{equation} [{\bf I}_t\beta_{jk},{\bf I}_t e_{ii}]=-(\delta_{ki}\alpha_{ji}+\delta_{ji}\alpha_{ki}),t=1,2,3, \end{equation}
\begin{equation} [{\bf I}_t\beta_{jk},{\bf I}_s e_{ii}]={\bf I}_t{\bf I}_s(\delta_{ki}\beta_{ji}+\delta_{ji}\beta_{ki}), t\neq s, \end{equation}
\begin{equation} [{\bf I}_t\beta_{jk},\beta_{l,n+1}]={\bf I}_t(\delta_{lk}\alpha_{j,n+1}+\delta_{jl}\alpha_{k,n+1}), t=1,2,3, \end{equation}
\begin{equation} [{\bf I}_t\beta_{jk},{\bf I}_t\alpha_{l,n+1}]=-(\delta_{lk}\beta_{j,n+1}+\delta_{jl}\beta_{k,n+1}), t=1,2,3, \end{equation}
\begin{equation} [{\bf I}_t\beta_{jk},{\bf I}_s\alpha_{l,n+1}]={\bf I}_t{\bf I}_s(\delta_{lk}\alpha_{j,n+1}+\delta_{jl}\alpha_{k,n+1}), t\neq s, \end{equation}
\begin{equation} [{\bf I}_te_{ii},{\bf I}_te_{ii}]=0, t=1,2,3,\end{equation}
\begin{equation} [{\bf I}_te_{ii},{\bf I}_se_{ii}]=2{\bf I}_t{\bf I}_se_{ii}, t\neq s, \end{equation}
\begin{equation}\label{ebetaI}[{\bf I}_te_{ii},\beta_{j,n+1}]={\bf I}_t(\delta_{ij}\alpha_{i,n+1}+\delta_{i,n+1}\alpha_{ij}), t=1,2,3, \end{equation}
\begin{equation} [{\bf I}_te_{ii},{\bf I}_t\alpha_{j,n+1}]=-(\delta_{ij}\beta_{i,n+1}-\delta_{i,n+1}\beta_{ij}), t=1,2,3,\end{equation}
\begin{equation} [{\bf I}_te_{ii},{\bf I}_s\alpha_{j,n+1}]={\bf I}_t{\bf I}_s(\delta_{ij}\alpha_{i,n+1}-\delta_{i,n+1}\alpha_{ij}), t\neq s, \end{equation}
\begin{equation} [\beta_{j,n+1},\beta_{k,n+1}]=\alpha_{jk}, \end{equation}
\begin{equation} [\beta_{j,n+1},{\bf I}_t\alpha_{k,n+1}]={\bf I}_t(-\beta_{jk}+2\delta_{jk}e_{n+1,n+1}), t=1,2,3, \end{equation}
\begin{equation} [{\bf I}_t\alpha_{j,n+1},{\bf I}_t\alpha_{k,n+1}]=\alpha_{jk}, t=1,2,3, \end{equation}
\begin{equation} [{\bf I}_t\alpha_{j,n+1},{\bf I}_s\alpha_{k,n+1}]={\bf I}_s{\bf I}_t(\beta_{jk}+2\delta_{jk}e_{n+1,n+1}),  t\neq s.\end{equation}
\end{pro}

By the above proposition, we can verify the following proposition.
\begin{pro}
The {\it Cartan  decomposition}  $ \mathfrak{sp}(n,1)=\mathfrak{k}\oplus\mathfrak{p}$ have the following properties:
\begin{eqnarray}\label{dek}\mathfrak{k}&=&{\rm span}\{\alpha_{jk},{\bf i}\beta_{jk},{\bf j}\beta_{jk},{\bf k}\beta_{jk},1\leq j<k\leq n, {\bf i}e_{ii},{\bf j}e_{ii},{\bf k}e_{ii},i=1,2\cdots n+1\},\\
 \label{dep}\mathfrak{p}&=&{\rm span}\{\beta_{j,n+1},\ {\bf i}\alpha_{j,n+1},\ {\bf j}\alpha_{j,n+1},\ {\bf k}\alpha_{j,n+1},\ 1\leq j\leq n\}, \end{eqnarray}
\begin{equation}\label{kprelation}[\mathfrak{k},\mathfrak{k}]\subset \mathfrak{k},\ [\mathfrak{k},\mathfrak{p}]\subset\mathfrak{p},\  [\mathfrak{p},\mathfrak{p}]\subset\mathfrak{k}.    \end{equation}
\end{pro}

\subsection{The Canonical Metric of ${\rm Sp}(n,1) $}
 \begin{defi}\label{defstba} The standard $\br$-vector space  basis for $ \mathfrak{sp}(n,1),$  denoted by $\mathfrak{B}$, consists of the following set of $ 2n^2+5n+3 $ matrices:
   \begin{itemize}
     \item $ \alpha_{jk},\ {\bf i}\beta_{jk},\ {\bf j}\beta_{jk},\ {\bf k}\beta_{jk},\ \ 1\leq j<k\leq n$; there are $2n^2-2n$ of these.
     \item  $\sqrt{2}{\bf i}e_{ii}, \sqrt{2}{\bf j}e_{ii}, \sqrt{2}{\bf k}e_{ii},\ \ i=1,2\cdots n+1$; there are $3n+3$ of these.
     \item   $\beta_{j,n+1},\  {\bf i}\alpha_{j,n+1},\ {\bf j}\alpha_{j,n+1},\ {\bf k}\alpha_{j,n+1},\ \ 1\leq j\leq n$; there are $4n$ of these.
   \end{itemize}
    \end{defi}
For $ X \in \mathfrak{sp}(n,1) $, the adjoint action of $ X $ is the $ \mathfrak{sp}(n,1)$-endomorphism defined by the Lie bracket
$$ {\rm ad} X(Y)=[X,Y].$$
We relabel the above standard  basis of $ \mathfrak{sp}(n,1)$ according to the order of sequence  as $e_1,\cdots,e_{2n^2+5n+3}$.  Let $C_{ij}\in \br^{2n^2+5n+3}$ be  the coefficients of ${\rm ad } e_i(e_j)$ represented by the basis.   That is
\begin{equation}\label{defcij} {\rm ad }\, e_i(e_j)=[e_i,e_j]=(e_1,e_2,\cdots,e_{2n^2+5n+3})C_{ij}.\end{equation}
 We mention that $C_{ij}$ can be read off from  Proposition \ref{baseprd}.

 Let $$X=\sum_{i=1}^{2n^2+5n+3}x_ie_i, \  x_i\in\br.$$ Then
\begin{equation}\label{adx} {\rm ad} X=\left(\sum_{i=1}^{2n^2+5n+3}x_iC_{i1},\cdots,\sum_{i=1}^{2n^2+5n+3}x_iC_{i,2n^2+5n+3}\right). \end{equation}
We note that ${\rm ad} X$ is a real square matrix of dimension $2n^2+5n+3$.

The Killing form on $ \mathfrak{sp}(n,1) $ is a symmetric bilinear form given by $$ B(X,Y)={\rm trace}({\rm ad}X{\rm ad}Y).$$
Let  $$Y=\sum_{i=1}^{2n^2+5n+3}y_ie_i,\   y_i\in\br.$$  Then
\begin{equation}\label{bxy}B(X,Y)=-8(n+2)\sum_{i=1}^{2n^2+n+3}x_iy_i+8(n+2)\sum_{i=2n^2+n+4}^{2n^2+5n+3}x_iy_i. \end{equation}
The Killing form enjoys the following important property:
\begin{equation}\label{skewb} B([X,Y],Z)+B(Y,[X,Z])=0,\  \mbox{for}\ \ X,Y,Z\in\mathfrak{sp}(n,1) .  \end{equation}
A positive definite inner product on $ \mathfrak{sp}(n,1)$ is  defined by
\begin{equation}\label{innprd} \langle X,Y\rangle=\left\{\begin{aligned}
 B(X,Y)\ \ \ \mbox{for}\ X,Y\in\mathfrak{p},\\
-B(X,Y)\ \  \mbox{for}\  X,Y\in\mathfrak{k},\\
0\ \ \  \  \ \  \mbox{otherwise}.\end{aligned}\right. \end{equation}
By identifying $ \mathfrak{sp}(n,1) $ with the tangent space at the identity of ${\rm Sp}(n,1)$, we can extend  $\langle\cdot,\cdot\rangle$ to a left invariant Riemannian metric over ${\rm Sp}(n,1).$  We denote this metric by $g$ and refer to it as the canonical metric for $ {\rm Sp}(n,1) $.

By (\ref{bxy}) and (\ref{innprd}), we have the following lemma.
\begin{lem}\label{disu}
For $ X, Y\in\mathfrak{B}$,
$$ \langle X,Y\rangle=\left\{\begin{aligned}
8(n+2)\ \ \ \ \if\ \ \ X=Y\\
0 \ \ \ \ {\rm otherwise}.\end{aligned}\right.$$
\end{lem}

\begin{cor}\label{cor1} The matrix representation for the canonical metric $g$ of  $ {\rm Sp}(n,1) $ is the square $2n^2+5n+3$ diagonal matrix
\begin{equation}\label{cmetric}\left(
      \begin{array}{cccc}
       8(n+2) & \ & \ & \  \\
        \ & 8(n+2) & \ & \  \\
         \ & \ & \ddots & \  \\
         \ & \ & \ & 8(n+2)  \\
      \end{array}
    \right).\end{equation}
    \end{cor}

\begin{defi}\label{qmetric}
Let $g$ be the canonical metric on ${\rm Sp}(n,1)$. The metric $\widetilde{g}$ on $ {\rm Sp}(n,1)$  is defined by \begin{equation}\label{scmetric} \widetilde{g}=\frac{1}{2(n+2)}g.\end{equation}
\end{defi}
We will show in Section \ref{volume} that   the metric $\widetilde{g}$  on ${\rm Sp}(n,1)$  induces holomorphic sectional curvature $-1$ on the quotient $ {\rm Sp}(n,1)/{\rm Sp}(n)\times {\rm Sp}(1).$

The canonical metric $g$  on a Lie algebra $\mathfrak{g}$ induces a norm given by $$ \|X\|=\langle X,X\rangle^\frac{1}{2}.$$
Let  \begin{equation}\label{nadx} N({\rm ad}X)=\sup\{\|{\rm ad}X(Y)\|:\  Y\in\mathfrak{g},\|Y\|=1\},\end{equation}
$$C_1=\sup\{N({\rm ad}X):\ X\in\mathfrak{p},\|X\|=1\}$$
and $$ C_2=\sup\{N({\rm ad}U): \ U\in\mathfrak{k},\|U\|=1\}.$$
The appendix to \cite{wan69} includes a table of the constants $C_1$ and $C_2$ for noncompact and nonexceptional Lie groups. The values for ${\rm Sp}(n,1)$ are $$ C_1=\frac{1}{\sqrt{2(n+2)}},\ C_2=\sqrt{2}C_1. $$
With respect to the scaled canonical metric  $\widetilde{g}$, we have \begin{equation}\label{c1c2}C_1=1,\  C_2=\sqrt{2}. \end{equation}

 \section{The Sectional Curvature  of $ {\rm Sp}(n,1)$}\label{sectcur}
  This section aims to obtain the bound of the sectional curvatures of
 ${\rm Sp}(n,1)$  with respect to the scaled  canonical metric.

\subsection{The connection  and curvature}
  By the fundamental theorem of Riemannian geometry,
  a connection   $\nabla $ on the tangent bundle of a manifold  can be expressed in terms of a left invariant metric  $ \langle , \rangle$  by the $ Koszul formula $.  For any left invariant vector fields $ X,Y,Z$,  we have
 \begin{equation}\label{nxy} \langle\nabla_X Y,Z\rangle=\frac{1}{2}\{\langle[X,Y],Z\rangle-\langle Y,[X,Z]\rangle-\langle X,[Y,Z]\rangle\}.\end{equation}
The {\it curvature tensor}   of a connection $\nabla$ is defined by
\begin{equation}\label{rxy}R(X,Y)Z=\nabla_X\nabla_YZ-\nabla_Y\nabla_XZ-\nabla_{[X,Y]}Z.\end{equation}
We mention that a connection is torsion free
\begin{equation}\label{torsion} \nabla_X Y-\nabla_Y X=[X,Y].\end{equation}

When a Lie group G is semisimple and compact, the canonical metric is the negative of the Killing
form and induces a biinvariant metric on G. The connection and curvature can be described in terms of the Lie bracket in a simple way \cite[Proposition 12 in Chapter 4]{Peter}.

When $G$ is semisimple and noncompact, a canonical metric is biinvariant only when restricted to
$K$, the maximal compact subgroup of G with Lie algebra $\mathfrak{k}$.  Adeboye and Wei have derived similar formulae  for the connection and curvature  for this case in  \cite[Proposition 3.3]{abe12} and   \cite[Proposition 2.7]{abe14} .

 We mention that in order to obtain better estimate, we will use some formulae of the connection  and curvature which are slight different from those in \cite{abe14}.  Those formulae  can be easily derived  by  the  properties of (\ref{skewb}),(\ref{nxy})-(\ref{torsion}) and the Jacobi identity.   For the convenience of the readers, we recall them as the following two propositions.

\begin{pro}{\rm  (\cite[Lemma 3.2]{abe12})}
Let $ U,V,W\in\mathfrak{k} $ and $ X,Y,Z\in\mathfrak{p}$. Then we have  the following equalities:
\begin{equation}\label{}\nabla_U V=\frac{1}{2}[U,V], \ \nabla_U X=\frac{3}{2}[U,X];\end{equation}
\begin{equation}\label{}\nabla_X Y=\frac{1}{2}[X,Y], \  \nabla_X U=-\frac{1}{2}[X,U].\end{equation}
\end{pro}

\begin{pro}\label{pro2.7}{\rm  (cf. \cite[Proposition 2.7]{abe14})}
Let $ U,V,W\in\mathfrak{k} $ and $ X,Y,Z\in\mathfrak{p}$. Then we have   the following equalities:
\begin{equation}\label{uvw}R(U,V)W=\frac{1}{4}[[V,U],W], \end{equation}
\begin{equation}\label{xyz}R(X,Y)Z=-\frac{7}{4}[[X,Y],Z], \end{equation}
\begin{equation}\label{xvy}R(X,V)Y=\frac{1}{4}[X,[V,Y]]+\frac{1}{4}[V,[X,Y]], \end{equation}
\begin{equation}\label{xvv}R(X,V)V=\frac{1}{4}[V,[X,V]], \end{equation}
\begin{equation}\label{xyv}R(X,Y)V=\frac{3}{4}[V,[X,Y]]. \end{equation}
 In  particular
\begin{equation}\label{uvwx}\langle R(U,V)W,X\rangle=0, \end{equation}
\begin{equation}\label{xyzu}\langle R(X,Y)Z,U\rangle=0, \end{equation}
\begin{equation}\label{uvvu}\langle R(U,V)V,U\rangle=\frac{1}{4}\|[U,V]\|^2, \end{equation}
\begin{equation}\label{xyyx}\langle R(X,Y)Y,X\rangle=-\frac{7}{4}\|[X,Y]\|^2, \end{equation}
\begin{equation}\label{uxxu}\langle R(U,X)X,U\rangle=\frac{1}{4}\|[U,X]\|^2. \end{equation}
\end{pro}
\begin{proof} Comparing  with \cite[Proposition 2.7]{abe14}), we only need to prove (\ref{xvy})-(\ref{xyv}).
Note that $$-\frac{1}{2}[X,[V,Y]]-\frac{1}{2}[V,[Y,X]]-\frac{1}{2}[Y,[X,V]]=0.$$  We obtain that
\begin{eqnarray*}
R(X,V)Y&=& \nabla_X\nabla_VY-\nabla_V\nabla_XY-\nabla_{[X,V]}Y\\
&=& \frac{3}{4}[X,[V,Y]]-\frac{1}{4}[V,[X,Y]]+\frac{1}{2}[Y,[X,V]]\\
&=& \frac{1}{4}[X,[V,Y]]+\frac{1}{4}[V,[X,Y]].
\end{eqnarray*}
Similarly we have
\begin{eqnarray*}
R(X,V)V&=&\nabla_X\nabla_VV-\nabla_V\nabla_XV
-\nabla_{[X,V]}V\\
&=& \frac{3}{4}[V,[X,V]]+\frac{1}{2}[[X,V],V]\\
&=& \frac{1}{4}[V,[X,V]]
\end{eqnarray*}
and
\begin{eqnarray*}
R(X,Y)V&=&\nabla_X\nabla_YV-\nabla_Y\nabla_XV
-\nabla_{[X,Y]}V\\
&=& -\frac{1}{4}[X,[Y,V]]+\frac{1}{4}[Y,[X,V]]-\frac{1}{2}[[X,Y],V]\\
&=& \frac{3}{4}[V,[X,Y]].
\end{eqnarray*}

\end{proof}

\begin{defi}
The  sectional curvature  of the planes  spanned by  $ X,Y\in\mathfrak{g} $ is denoted and defined by
\begin{equation}\label{dfnsc}K(X,Y)=\frac{\langle R(X,Y)Y,X\rangle}{\|X\|^2\|Y\|^2-\langle X,Y\rangle^2}.\end{equation}
\end{defi}

\begin{pro}\label{basecb}
 The sectional  curvature  of  ${\rm Sp}(n,1)$ with  respect  to  the  metric  $\widetilde{g}$  at  the  planes  spanned  by standard  basis elements is bounded  above by  $\frac{1}{2}.$
\end{pro}
\begin{proof}
Since the basis elements are mutually orthogonal, the sectional curvature at the plane spanned by any distinct elements
$ X,Y\in\mathfrak{B} $ is given by $$ K(X,Y)=\frac{\langle R(X,Y)Y,X\rangle}{\|X\|^2\|Y\|^2}.$$
By (\ref{uvvu})-(\ref{uxxu}) and Proposition \ref{baseprd}, the largest sectional curvature spanned by basis directions are the planes spanned by $ {\bf i}\sqrt{2}e_{ii},{\bf j}\sqrt{2}e_{ii};{\bf i}\sqrt{2}e_{ii},{\bf k}\sqrt{2}e_{ii}$ or ${\bf j}\sqrt{2}e_{ii},{\bf k}\sqrt{2}e_{ii}.$  The largest sectional curvature is given by
\begin{equation} K({\bf i}\sqrt{2}e_{ii},{\bf j}\sqrt{2}e_{ii})=\frac{\frac{1}{4}\|[{\bf i}\sqrt{2}e_{ii},{\bf j}\sqrt{2}e_{ii}]\|^2}{\|{\bf i}\sqrt{2}e_{ii}\|^2\|{\bf j}\sqrt{2}e_{ii}\|^2}
=\frac{1}{4}\frac{\|4{\bf k}e_{ii}\|^2}{16}=\frac{1}{2}. \end{equation}
\end{proof}

 Let $$R(X_1,X_2,X_3,X_4):=\langle  R(X_1,X_2)X_3,X_4 \rangle.$$  We recall the following facts \cite[Page 33]{Peter}  for any left invariant vector fields $X_1,X_2,X_3,X_4$:
\begin{equation}\label{proc1}R(X_1,X_2,X_3,X_4)=-R(X_2,X_1,X_3,X_4)=R(X_2,X_1,X_4,X_3),\end{equation}
\begin{equation}\label{proc2}R(X_1,X_2,X_3,X_4)=R(X_3,X_4,X_1,X_2).\end{equation}

\begin{pro}\label{secub}
 The  sectional  curvatures  of  ${\rm Sp}(n,1)$ with  respect  to  $\widetilde{g}$ are  bounded  above by
 $\frac{3+4\sqrt{2}}{2}.$
\end{pro}

\begin{proof}
Let $ U,V \in \mathfrak{k}$ and  $X,Y \in \mathfrak{p}$.  Then by the above properties  (\ref{proc1}),(\ref{proc2}) and (\ref{uvwx}),(\ref{xyzu}) we can reduce the  sixteen items of the expansion  of   $\langle R(X+U,Y+V)(Y+V),X+U \rangle $ to eight items. That is
\begin{eqnarray*} \langle R(X+U,Y+V)(Y+V),X+U \rangle &=& \langle R(X,Y)Y,X \rangle +\langle R(U,V)V,U \rangle +\langle  R(U,Y)Y,U \rangle \\
&+&\langle R(X,V)V,X \rangle +2 \langle R(X,Y)V,U \rangle +2 \langle R(X,V)Y,U \rangle.
\end{eqnarray*}
Assume that $ \|U+X \|=1,\|V+Y\|=1$ and $\langle U+X,V+Y\rangle=0.$
Let $c_V,c_U,c_X,c_Y$ be real numbers  satisfying
\begin{equation}\label{coef}\|c_VV\|=\|c_UU\|=\|c_XX\|=\|c_YY\|=1.\end{equation}
 By our assumption, it obvious that $$|c_V|,|c_U|,|c_X|,|c_Y|\geq 1.$$
It follows from (\ref{nadx}) that $$\|{\rm ad}V(U)\|=\|[V,U]\|\leq \|[c_VV,c_U U]\|=\|{\rm ad}c_VV(c_UU)\|\leq C_2.$$
Similarly we have  $$\|{\rm ad}V(X)\|\leq C_2, \|{\rm ad}X(U)\|\leq C_1,  \|{\rm ad}X(Y)\|\leq C_1.$$
Therefore  by (\ref{uvvu}), (\ref{uxxu}), (\ref{xvv})  and  (\ref{c1c2})  we have
$$  \langle R(U,V)V,U \rangle =\frac{1}{4} \|{\rm ad}V(U)\|^2 \leq \frac{1}{4}C_2^2=\frac{1}{2},$$
$$  \langle R(U,Y)Y,U \rangle =\frac{1}{4} \|{\rm ad}Y(U)\|^2 \leq \frac{1}{4}C_1^2=\frac{1}{4}$$
and
$$\langle R(X,V)V,X \rangle = \langle \frac{1}{4}[V,[X,V]],X \rangle=\frac{1}{4} \|{\rm ad}X(V)\|^2 \leq \frac{1}{4}C_1^2=\frac{1}{4}.$$
By (\ref{xyv})  we have
\begin{eqnarray*}
\langle R(X,Y)V,U\rangle&=&\frac{3}{4}\langle[X,Y],[U,V]\rangle\\
&\leq&  \frac{3}{4}\|[X,Y]\|\,\|[U,V]\|\\
&=& \frac{3}{4}\|{\rm ad}X(Y)\|\|{\rm ad}U(V)\|\\
&\leq& \frac{3C_1C_2}{4}=\frac{3\sqrt{2}}{4}.\end{eqnarray*}
By (\ref{xvy})  we have \begin{eqnarray*}  \langle R(X,V)Y,U\rangle &=&\frac{1}{4}\langle[V,Y],[X,U]\rangle+\frac{1}{4}\langle[X,Y],[U,V]\rangle\\
&\leq&  \frac{1}{4}\Big[\|[X,U]\|\,\|[Y,V]\| + \frac{1}{4}\|[X,Y]\|\,\|[U,V]\|\Big]\\
&=& \frac{1}{4}\Big[\|{\rm ad}X(U)\|\|{\rm ad}Y(V)\| +\frac{1}{4} \|{\rm ad}X(Y)\|\|{\rm ad}U(V)\|\Big]\\
&\leq& \frac{1}{4}(C_1^2+C_1C_2)=\frac{1+\sqrt{2}}{4}.\end{eqnarray*}
Noting that
\begin{eqnarray*}
\langle R(X,Y)Y,X\rangle = -\frac{7}{4}\|[X,Y]\|^2\leq 0,
\end{eqnarray*}
we obtain that the sectional curvatures of  of  ${\rm Sp}(n,1)$ with  respect  to  $\widetilde{g}$ are  bounded  above by
$$\frac{1}{2}+2\cdot\frac{1}{4}+2\cdot\frac{1+\sqrt{2}}{4}+2\cdot\frac{3\sqrt{2}}{4}=\frac{3+4\sqrt{2}}{2}.$$
\end{proof}

\section{The volume of quaternionic hyperbolic orbifolds}\label{volume}

This section  contains the proof of Theorem \ref{mainthm} with the similar route map employed  in \cite{abe14}. First, we construct a   Riemannian submersion form the  quotient ${\rm Sp}(n,1)/\Gamma$ to the quotient  ${\bf H}_\bh^n/\Gamma$.  With this Riemannian submersion, we can employ Wang's result \cite[Theorem 5.2]{wan69} to produce an inscribed ball of radius $\frac{R_{{\rm Sp}(n,1)}}{2}$ in ${\bf H}_\bh^n/\Gamma$ and obtain the lower bound by a comparison theorem of Gunther \cite[Theorem 3.101]{gal90}.

\subsection{Riemannian Submersions}

 \begin{defi}Let $ (M,g) \ and \ (N,h)$ be Riemannian manifolds and $ q:M\rightarrow N$ a surjective submersion. For each
point $x\in M$,  the tangent space $T_xM$  can be decomposed into the orthogonal direct sum
$$T_xM=({\rm Ker}\  dq)^\perp_x+({\rm Ker}\  dq)_x.$$
 The map $q$ is said to be a  Riemannian submersion if $$ g(X,Y)=h(dqX,dqY),  \forall X,Y\in ({\rm Ker}\  dq)^\perp_x \ for \ some  \ x\in M. $$
 \end{defi}

Let  $ X,Y $ be orthonormal vector fields on $ N $ and let $\widetilde{X},\widetilde{Y}$ be their horizontal lifts to $ M .$   O'Neill's formula\cite[Page 127]{gal90} relates the sectional curvature of the base space of a Riemannian submersion with that of the total space
\begin{equation}\label{neformula}K_b(X,Y)=K_t(\widetilde{X},\widetilde{Y})+\frac{3}{4}\|[\widetilde{X},\widetilde{Y}]^\perp\|^2, \end{equation}
where $Z^\perp $ represents the vertical component of $Z$.

\begin{defi}
 Let $\bJ$ be the complex structure on $\mathfrak{p}$ such that
\begin{equation}\label{qcmstr}\bJ X= \sum_{j=1}^{n}(a_{2j}\beta_{j,n+1}-a_{1j}{\bf i}\alpha_{j,n+1}-a_{4j}{\bf j}\alpha_{j,n+1}+a_{3j}{\bf j}\alpha_{j,n+1}), \end{equation}
for $X=\sum_{j=1}^{n}(a_{1j}\beta_{j,n+1}+a_{2j}{\bf i}\alpha_{j,n+1}+a_{3j}{\bf j}\alpha_{j,n+1}+a_{4j}{\bf k}\alpha_{j,n+1})\in \mathfrak{p}$.
\end{defi}
We remind that
 $$B(\bJ X,\bJ Y)=B(X,Y).$$   This implies that the complex structure preserves the Killing form $B(X,Y)$.

\begin{pro}\label{holxy}
Let\begin{equation}\label{holx} X= \sum_{j=1}^{n}(a_{1j}\beta_{j,n+1}+a_{2j}{\bf i}\alpha_{j,n+1}+a_{3j}{\bf j}\alpha_{j,n+1}+a_{4j}{\bf k}\alpha_{j,n+1})\end{equation} and
 \begin{equation}\label{holy} Y= \bJ X= \sum_{k=1}^{n}(a_{2k}\beta_{k,n+1}-a_{1k}{\bf i}\alpha_{k,n+1}-a_{4k}{\bf j}\alpha_{k,n+1}+a_{3k}{\bf k}\alpha_{k,n+1}),\end{equation}
 where
 $$  \sum_{j=1}^{n}(a_{1j}^2+a_{2j}^2+a_{3j}^2+a_{4j}^2)=\frac{1}{4}.$$
  Then
 $$\|[X,Y]\|^2=1.$$
\end{pro}

 \begin{proof}
 It is obvious that $ \|X \|=1, \|Y \|=1$ and $\langle X, Y\rangle=0$.
By Proposition \ref{baseprd} we have
\begin{eqnarray*}
[X,Y]&=& \sum_{j\neq k}\Big\{(a_{1j}a_{2k}-a_{2j}a_{1k}-a_{3j}a_{4k}+a_{4j}a_{3k})\alpha_{jk}\\
&+&(a_{1j}a_{1k}+a_{2j}a_{2k}-a_{3j}a_{3k}-a_{4j}a_{4k})\mathbf{i}\beta_{jk}\\
&+&(a_{1j}a_{4k}+a_{2j}a_{3k}+a_{3j}a_{2k}+a_{4j}a_{1k})\mathbf{j}\beta_{jk}\\
&+&(-a_{1j}a_{3k}+a_{2j}a_{4k}-a_{3j}a_{1k}+a_{4j}a_{2k})\mathbf{k}\beta_{jk}\Big\}\\
&+& \sum_{j=1}^{n}\Big\{\sqrt{2}(a_{1j}^2+a_{2j}^2-a_{3j}^2-a_{4j}^2)\sqrt{2}{\bf i}e_{jj}\\
&+&2\sqrt{2}(a_{1j}a_{4j}+a_{2j}a_{3j})\sqrt{2}{\bf j}e_{jj}\\
&+&2\sqrt{2}(a_{2j}a_{4j}-a_{1j}a_{3j})\sqrt{2}{\bf k}e_{jj}\\
&-&\sqrt{2}(a_{1j}^2+a_{2j}^2+a_{3j}^2+a_{4j}^2)\sqrt{2}{\bf i}e_{n+1,n+1}\Big\}.
\end{eqnarray*}
Hence
\begin{eqnarray*}
\frac{\|[X,Y]\|^2}{4}& =& 4\sum_{j< k}\Big\{(a_{1j}a_{2k}-a_{2j}a_{1k}-a_{3j}a_{4k}+a_{4j}a_{3k})^2\\
&+&(a_{1j}a_{1k}+a_{2j}a_{2k}-a_{3j}a_{3k}-a_{4j}a_{4k})^2\\
&+&(a_{1j}a_{4k}+a_{2j}a_{3k}+a_{3j}a_{2k}+a_{4j}a_{1k})^2\\
&+&(-a_{1j}a_{3k}+a_{2j}a_{4k}-a_{3j}a_{1k}+a_{4j}a_{2k})^2\Big\}\\
&+& \sum_{j=1}^{n}\Big\{2(a_{1j}^2+a_{2j}^2-a_{3j}^2-a_{4j}^2)^2+8(a_{1j}a_{4j}+a_{2j}a_{3j})^2+8(a_{2j}a_{4j}-a_{1j}a_{3j})^2\Big\}+ \frac{1}{8}\\
&=& 4\sum_{j< k}(a_{1j}^2+a_{2j}^2+a_{3j}^2+a_{4j}^2)(a_{1k}^2+a_{2k}^2+a_{3k}^2+a_{4k}^2)+\sum_{j=1}^{n}\Big\{2(a_{1j}^2+a_{2j}^2+a_{3j}^2+a_{4j}^2)^2\Big\}+ \frac{1}{8}\\
&=&2 (\sum_{j=1}^{n}(a_{1j}^2+a_{2j}^2+a_{3j}^2+a_{4j}^2))^2+ \frac{1}{8}=\frac{1}{4}.
\end{eqnarray*}
 \end{proof}

\begin{pro}\label{prokb1}
Consider the quotient map
\begin{equation}\label{quopi}\pi :{\rm Sp}(n,1)\to {\rm Sp}(n,1)/{{\rm Sp}(n)\times {\rm Sp}(1)}.\end{equation}
Then the restriction of the inner product $ \langle X,Y \rangle$,  defined on  $\mathfrak{sp}(n,1)=\mathfrak{k} \oplus \mathfrak{p}$,  to  $$d_e\pi (\mathfrak{p})=T_{\pi(e)}{\rm Sp}(n,1)/{{\rm Sp}(n)\times {\rm Sp}(1)},$$  induces a Riemannian metric on the quotient space.  That is  the map $ \pi $ is a Riemannian submersion.
\end{pro}

\begin{proof}
We need to  show that  ${\rm Sp}(n,1)/{{\rm Sp}(n)\times {\rm Sp}(1)}$ has constant holomorphic sectional  curvature $-1$  with the restriction of the scaled canonical metric $ \widetilde{g}$.

 Let $X$ represent both a unit vector field on ${\rm Sp}(n,1)/{{\rm Sp}(n)\times {\rm Sp}(1)}$ as well as its horizontal lift.
Let $X$ and $Y=\bJ X$  be given by (\ref{holx}) and (\ref{holy}).
By (\ref{dfnsc}) we have
 $$ K_t(X,Y)=\langle R(X,Y)Y,X\rangle= -\frac{7}{4}\|[X,Y]\|^2.$$
 Since $[X,Y]\in \mathfrak{k}$, $[X,Y]^{\bot}=[X,Y]$. It follows from Proposition \ref{holxy} and   O'Neill's formula (\ref{neformula}) that
 $$K_b(X,\bJ X)=K_b(X,Y)= -\|[X,Y]\|^2=-1.$$
 This implies that $ \pi $ is a Riemannian submersion from $ {\rm Sp}(n,1)$ to the quaternionic  hyperbolic $n$-space  $\mathbf{H}^n_{\mathbb{H}}$.
\end{proof}

\subsection{Wang and Gunther's results}
The following result gives Wang's quantitative version of the well-known result of Kazhdan-Marglis \cite{kama68}.
\begin{lem}{\rm(\cite[Theorem 5.2]{wan69})}\label{wanresult} Let $ G $ be a semisimple Lie group without compact factor, let $id$ be the identity of $G $, let $\rho $ be the distance function derived from a canonical metric, and let
$$ B_G=\{x\in G:\rho(id,x)\leq R_G\}.$$
 Then  for  any  discrete  subgroup  $\Gamma $ of $G$, there  exists  $g\in G$  such  that  $B_G\cap \ g\Gamma g ^{-1}={id}.$
\end{lem}
 Wang also showed that number $R_G$  is less than the injectivity radius of $G$.   Consequently, the volume of the fundamental domain of any discrete subgroup $\Gamma$ of $ G ,$ when viewed as a group of left translations of $G$, is bounded from below by the volume of a $\rho$ -ball of radius $\frac{R_G}{2}$.

Since ${\rm Sp}(n,1)$ is a semisimple Lie group without compact factor.  Let $C_1$ and  $C_2$  be given by (\ref{c1c2}). By \cite{wan69}  the number $R_{{\rm Sp}(n,1)} $ is  the least positive zero of the real-valued function
\begin{equation}  F(t)=\exp C_1t-2\sin C_2t-\frac{C_1t}{\exp C_1t-1}. \end{equation}
That is  \begin{equation}\label{rgsp}R_{{\rm Sp}(n,1)}\approx0.228\ldots\end{equation}

\vspace{1.5mm}
Let $V(d, k, r)$ denote the volume of a ball of radius $r$ in the
complete simply connected Riemannian manifold of dimension $d$ with constant curvature $k$.
In \cite{abe12},  Adeboye and Wei  obtained the following formula
\begin{equation}\label{vdkr} V(d,k,r)=\frac{2(\frac{\pi}{k})^{\frac{d}{2}}}{\Gamma(\frac{d}{2})}\int^{\min(r\sqrt{k},\ \pi)}_0\sin^{d-1}\rho d\rho.\end{equation}

\vspace{2mm}
We recall the following  comparison theorem of Gunther.

 \begin{lem}{\rm  (\cite[Theorem 3.101]{gal90})}\label{gunresult}
 Let  $M$  be  a  complete  Riemannian  manifold  of  dimension  $d$. For  $ m \in M$,  let  $ B_{m}(r)$ be  a  ball  which  does  not  meet  the  cut-locus  of  $m$.
If  the  sectional  curvatures  of  $M$  are bounded above by  a  constant  $b$, then
$$ {\rm Vol}[B_{m}(r)] \geq V(d,b,r).$$
\end{lem}

\subsection{The proof of main result}
In order to prove our main result, we need the following four lemmas.
The following two lemmas have been proved in \cite{abe12,abe14}.
\begin{lem}{\rm (\cite[Lemma 3.4]{abe14})}\label{togeo1}
Let  G  be  a  semisimple  Lie  group  and  $\mathfrak{g}$  be  its  Lie  algebra,with  Cartan    decomposition  $\mathfrak{g}=\mathfrak{k}\oplus\mathfrak{p}$. Let  $K$  be  the  maximal  compact  subgroup  of  $G$  with  Lie  algebra $\mathfrak{k}$. Then , with  respect  to  the  canonical  metric, $K$  is  totally  geodesic  in  $G$.
\end{lem}

\begin{lem}{\rm (\cite[Lemma 3.5]{abe14})}\label{partvol}
 Let  $K\rightarrow M \stackrel{q}{\longrightarrow} N$ denote  a  Riemannian submersion and  $K$  is a  compact  and  totally  geodesic  submanifold  of  $M$. Then  for any subset $ Z\subset N$,
 $${\rm Vol}[q^{-1}(Z)]=Vol[Z]\cdot {\rm Vol}[K].$$
\end{lem}

The following simple lemma is the main tool we use to produce our estimate.
\begin{lem}\label{volform}
Let  $\Gamma$  be  a  discrete  subgroup  of   ${\rm Sp}(n,1)$,   then
$$ {\rm Vol}[{\rm Sp}(n,1)/\Gamma]\geq V(d_0,k_0,r_0),$$
where  $d_0=2n^2+5n+3, k_0=\frac{3+4\sqrt{2}}{2}$  and $\ r_0=0.114.$
\end{lem}

\begin{proof}
The inequality follows  from Lemma \ref{wanresult} and \ref{gunresult}. The values of $d_0, k_0$ and
$r_0$  follow from Definition \ref{defstba}, Proposition \ref{secub} and (\ref{rgsp}), respectively.
\end{proof}

\begin{lem}\label{volsp}
With respect to the metric $ \widetilde{g} $  given by  (\ref{scmetric}), $$ {\rm Vol}[{\rm Sp}(n)\times {\rm Sp}(1)]=\frac{2^n(\pi)^{(n^2+n+\frac{3}{2})}\Gamma(\frac{4n+1}{2})}{(2n+1)!(2n-1)!(2n-3)!\cdots!5!3!1!~}.$$
\end{lem}
\begin{proof} The volumes of the classical compact groups are given explicitly in \cite[Chapter 9]{gilmore}.  The volume formulae  with respect to the metric   $ \widetilde{g} $  are
$${\rm Vol}[{\rm Sp}(n+1)]=\frac{2^{n+1}(\pi)^{(n+1)(n+2)}}{(2n+1)!(2n-1)!(2n-3)!\cdots!5!3!1!~}$$
and
$${\rm Vol}[{\rm Sp}(n+1)/{{\rm Sp}(n)\times {\rm Sp}(1)}]=\frac{2\,(\pi)^{\frac{4n+1}{2}}}{\Gamma(\frac{4n+1}{2})}.$$
Hence
$$ {\rm Vol}[{\rm Sp}(n)\times {\rm Sp}(1)]=\frac{2^n(\pi)^{(n^2+n+\frac{3}{2})}\Gamma(\frac{4n+1}{2})}{(2n+1)!(2n-1)!(2n-3)!\cdots!5!3!1!~}.$$
\end{proof}

 We now are ready to give a proof of Theorem \ref{mainthm}, which for convenience is restated below.

\vspace{2mm}

{\bf Theorem 1.3}
\quad {\it The volume of a quaternionic  hyperbolic $ n$-orbifold is bounded below by $\mathcal{Q}(n)$,  an  explicit constant depending only on dimension, given by
$$ \mathcal{Q}(n)=\frac{\pi^{\frac{3n}{2}}(2n+1)!(2n-1)!\cdots!5!3!1!}{2^{n-1}
\Gamma(\frac{2n^2+5n+3}{2})\Gamma(\frac{4n+1}{2}) (\frac{3+4\sqrt{2}}{2})^{\frac{2n^{2}+5n+3}{2}}} \int^{0.2372}_0(\sin\rho)^{2n^2+5n+2} d\rho.$$ }

\begin{proof}
 By Proposition \ref{prokb1}, the holomorphic sectional curvature of ${\bf H}_\bh^n$ is normalized to be $-1$ by the metric $\widetilde{g}$ given by (\ref{scmetric}). Let $Q$ be a quaternionic hyperbolic $n$-orbifold given by $$Q={\bf H}_\bh^n/\Gamma=[{\rm Sp}(n,1)/{\rm Sp}(n)\times {\rm Sp}(1)]/\Gamma.$$   Then the quotient map $\pi$ given by (\ref{quopi}) induces another Riemannian submersion
 $$\pi^{\prime}: {\rm Sp}(n,1)/\Gamma \to Q.$$
 The fibers of $\pi^{\prime}$ on the smooth points of $Q$ are totally geodesic embedded copies of ${\rm Sp}(n)\times {\rm Sp}(1).$
  By  Lemmas \ref{togeo1}, \ref{partvol}  and \ref{volform}, we have
$$V(d_0,k_0,r_0)\leq {\rm Vol}[{\rm Sp}(n,1)/\Gamma]\leq {\rm Vol}[\pi^{-1}(Q)]={\rm Vol}[Q]\cdot {\rm Vol}[{\rm Sp}(n)\times {\rm Sp}(1)].$$
Hence
$$ {\rm Vol}[Q] \ge   \frac{V(d_0,k_0,r_0)}{{\rm Vol}[{\rm Sp}(n)\times {\rm Sp}(1)]}=\mathcal{Q}(n).$$
The proof follows from Lemma \ref{volsp} and (\ref{vdkr}).
\end{proof}

\section{Slight improvement  for real and complex cases}\label{realcomplex}

By similar way of  Proposition \ref{secub}, we can  reestimate the sectional curvatures of ${\rm SO_o}(n,1)$ and
${\rm SU}(n,1)$.

\begin{pro}\label{secub-r-c}
 \begin{itemize}
   \item[(1)]  The  sectional  curvatures $k$  of  ${\rm SO_o}(n,1)$  with  respect  to  $\widetilde{g}$ given by \cite[Definition 1.5]{abe12} are  bounded  above by 
$\left\{
  \begin{array}{ll}
   \frac{1}{4}, & \hbox{when}\   n=2 \ \  (see \ \cite{abe12});  \\
    \frac{13}{4}, & \hbox{when}\ n=3; \\
    \frac{3+4\sqrt{2}}{2}, & \hbox{when}\   n\geq 4.
  \end{array}
\right.$

   \item[(2)]  The  sectional  curvatures  of  ${\rm SU}(n,1)$ with  respect  to  $\widetilde{g}$ given by Definition 2.6 in  \cite{abe14} are  bounded  above by
 $\frac{13}{4}.$
 \end{itemize}
 \end{pro}

In light of the above proposition, we can  slight improve the main result Theorem 0.1 in \cite{abe12,abe14} as  followings.

\begin{thm}\label{imabe0812}
\begin{itemize}
 \item[(1)] The volume of a real hyperbolic $n$ -orbifold is bounded below by $\mathcal{R}(n)$,  an  explicit constant depending only on dimension, given by
$$ \mathcal{R}(n)=\frac{2^{\frac{6-n}{4}}\pi^{\frac{n}{4}}(n-2)!(n-4)!\cdots 1}{(3+4\sqrt{2})^{\frac{n^2+n}{4}}
\Gamma(\frac{n^2+n}{4})} \int^{0.2372}_0(\sin\rho)^{\frac{n^2+n-2}{2}} d\rho.$$
 \item[(2)] The volume of a complex  hyperbolic $ n$ -orbifold is bounded below by $\mathcal{C}(n)$,  an  explicit constant depending only on dimension, given by
$$ \mathcal{C}(n)=\frac{2^{n^2+n+1}\pi^{\frac{n}{2}}(n-1)!(n-2)!\cdots!3!2!1!}{(13)^{\frac{n^2+2n}{2}}
\Gamma(\frac{n^2+2n}{2})} \int^{0.2497}_0(\sin\rho)^{n^2+2n-1} d\rho.$$
\end{itemize}
\end{thm}

\vspace{2mm}
The following is a table of the lower bound for the volume of hyperbolic $n$-orbifolds in \cite{abe12,abe14} and this paper for some cases of $n\leq 4$ (by software  Matlab of version R2009b).
$$\begin{tabular}{c|cc|ccc}
               \hline
                            &\multicolumn{2}{c}\   Results in [2,3]\quad \quad \quad \ \  \vline &\multicolumn{3}{r} \ Results of this paper\quad \quad \quad\quad \quad \quad \ \ \  \\
              \hline
               $n$ & $\mathcal{R}(n)$  & $\mathcal{C}(n)$  & $\mathcal{R}(n)$& $\mathcal{C}(n)$  & $\mathcal{Q}(n)$\\
               \hline
               $1$ &  & $0.00168$ &   & 0.00175  &  $3.6221\times 10^{-11}$ \\
                \hline
             $2$  &$0.00125$ & $2.9180\times 10^{-9}$ &  & $4.1822\times 10^{-9}$ &  $5.3637\times 10^{-25}$ \\
                \hline
               $3$ & $2.4583\times 10^{-7}$ & $3.6324\times 10^{-18}$ &  $2.8073\times 10^{-7}$ & $1.1556\times 10^{-17}$ &  \\
                  \hline
                $4$ & $3.1469\times 10^{-13}$  & $2.2347\times 10^{-30}$ & $4.0019\times 10^{-13}$ &  $3.7865\times 10^{-29}$  & \\
               \hline
             \end{tabular}
$$

\vspace{3mm}
{\bf Acknowledgements}\quad   We would like to thank John R. Parker and Ilesanmi Adeboye for reading this paper carefully and useful suggestions.   This work was  supported by NSF of Guangdong Province (2015A030313644)  and  State  Scholarship  Council of China,  and was completed when the first author was an Academic Visitor to Durham University. He would like to thank  Department of Mathematical Sciences  for its hospitality.

\end{document}